\documentclass[10pt,a4pwide,leqno]{amsart}
\usepackage{amsmath}
\usepackage{amssymb}
\usepackage{amsfonts}
\usepackage{amsthm}
\usepackage{mathrsfs}
\usepackage{dsfont}
\usepackage{mathtools}
\usepackage{todonotes}

\usepackage[backref=page]{hyperref}
\usepackage{color}

\newcommand*{\mailto}[1]{\href{mailto:#1}{\nolinkurl{#1}}}

\newcommand{\eps}{\varepsilon}

\newcommand{\px}{\partial_x}
\newcommand{\py}{\partial_y}
\newcommand{\pxx}{\partial_{xx}^2}

\newcommand{\pt}{\partial_t}

\newcommand{\dx}{\, dx}
\newcommand{\dtx}{\,d\tilde x}
\newcommand{\dy}{\, dy}
\newcommand{\dz}{\, dz}
\newcommand{\dt}{\, dt}
\newcommand{\ds}{\, ds}

\newcommand{\Grad}{\nabla}
\DeclareMathOperator{\Div}{div}
\newcommand{\dotW}{\dot{W}}
\newcommand{\dW}{\,dW}

\newcommand{\cS}{\mathcal{S}}

\newcommand{\tx}{\tilde x}

\newcommand{\cK}{\mathcal{K}}
\newcommand{\cU}{\mathcal{U}}
\newcommand{\cW}{\mathcal{W}}

\newcommand{\Dp}{\mathcal{D}^\prime}

\renewcommand{\d}{\mathrm{d}}

\newcommand{\sgn}{\operatorname{sgn}}
\newcommand{\supp}{\operatorname{supp}\,}

\newcommand{\loc}{\operatorname{loc}}

\newcommand{\R}{\mathbb{R}}
\newcommand{\N}{\mathbb{N}}

\newcommand{\cA}{\mathcal{A}}
\newcommand{\cD}{\mathcal{D}}

\newcommand{\cF}{\mathcal{F}}
\newcommand{\cL}{\mathcal{L}}
\newcommand{\ue}{u_\eps}

\newcommand{\seq}[1]{\left\{#1\right\}}

\newcommand{\Seq}[1]{\left\{#1\right\}}

\newcommand{\set}[1]{\left\{#1\right\}}

\newcommand{\Set}[1]{\left\{#1\right\}}

\newcommand{\abs}[1]{\left|#1\right|}
\newcommand{\absb}[1]{\bigl|#1\bigr|}
\newcommand{\norm}[1]{\left\|#1\right\|}

\newcommand{\bjd}{J_\delta}
\newcommand{\bjdh}{J_{\frac{\delta}{2}}}

\newcommand{\bjn}{J_\nu}

\newcommand{\signb}[1]{\operatorname{sign}\bigl(#1\bigr)}

\newcommand{\todelta}{\xrightarrow{\delta\downarrow 0}}

\newcommand{\ton}{\xrightarrow{n\uparrow \infty}}

\DeclareMathOperator*{\Ex}{\Bbb{E}}

\theoremstyle{theorem}
\newtheorem{theorem}{Theorem}[section]
\newtheorem{proposition}[theorem]{Proposition}
\newtheorem{lemma}[theorem]{Lemma}

\theoremstyle{definition}
\newtheorem{definition}[theorem]{Definition}
\theoremstyle{theorem}
\newtheorem{remark}[theorem]{Remark}

\numberwithin{equation}{section}

\title[Stochastic conservation laws]
{Quantitative compactness estimates for stochastic conservation laws}

\author[Karlsen]{K. H. Karlsen}
\address[Kenneth H. Karlsen]{Department of Mathematics\\
University of Oslo\\
NO-0316 Oslo\\ Norway}
\email{\mailto{kennethk@math.uio.no}}

\date{\today}

\begin{document}

\begin{abstract}
We present a quantitative compensated 
compactness estimate for stochastic conservation laws, 
which generalises a previous result of 
Golse \& Perthame \cite{Golse:2013aa} for deterministic equations. 
With a stochastic modification of Kru{\v{z}}kov's 
interpolation lemma, this estimate provides 
bounds on the rate at which a sequence 
of vanishing viscosity solutions becomes compact.
\end{abstract}

\maketitle

\section{Introduction}\label{sec:intro}
During the last decade many authors investigated the well-posedness of 
hyperbolic conservation laws perturbed by stochastic source terms, see 
\cite{Debussche:2010fk,Feng:2008ul} and the list 
of references in \cite{Frid:2021us}. 
The initial-value problem for these It\^{o}-type 
SPDEs take the form
\begin{equation}\label{eq:scl}
	\begin{split}
		& \pt u + \Div \! f(x,u)
		=\sigma(x,u)\dotW(t) + R(x,u),
		\quad (t,x)\in (0,T)\times \R^d,\\
		& u(0,x)=u_0(x), \quad x\in \R^d,
	\end{split}	
\end{equation}
where $f=(f_1,\ldots,f_d)$ is the flux vector, 
$R$ is the ``deterministic" source term, 
$u_0\in L^\infty(\R^d)$ is the initial 
function, and $T>0$ is a final time.  
The term $\sigma\dotW(t)$ is a stochastic forcing term, 
where $W$ is a cylindrical Wiener 
process \cite{DaPrato:2014aa} with noise amplitude $\sigma$. 
We will refer to the SPDEs \eqref{eq:scl} 
as stochastic conservation laws. Stochastic conservation laws 
are used to model a wide variety of physical systems 
that are subject to random fluctuations and have 
wave-propagating behavior. 

We fix a stochastic basic $\cS$ 
consisting of a complete probability space $(\Omega,\cF,P)$, 
and a complete right-continuous filtration 
$\Seq{\cF_t}_{t\in [0,T]}$. The solution $u$, the Wiener process $W$, 
and all other relevant processes, are always understood as defined 
on $\cS$ and to be appropriately measurable with respect to the 
filtration $\Set{\cF_t}_{t\in [0,T]}$. We refer 
to \cite[Pages 38, 40]{Frid:2021us} for precise 
regularity and growth assumptions on $f$, $\sigma$, $R$. 
For a precise definition of entropy/kinetic solutions and a 
corresponding well-posedness theorem, see 
\cite[Section 5]{Frid:2021us}. Under the assumptions 
that $R\equiv 0$ and $f=f(u)$, we
refer to the original works \cite{Feng:2008ul} (on $\R^d$) and 
\cite{Debussche:2010fk} (on $\Bbb{T}^d$).
 
In this paper, we are interested in deriving 
quantitative estimates 
that can be used to prove the 
convergence in $L^1_{\loc}$ of sequences 
$\Seq{u_n}_{n\in \N}$ of approximate 
solutions to \eqref{eq:scl}. As a concrete example, 
consider the parabolic SPDE
\begin{equation}\label{eq:scl-viscous}
	\begin{split}
		& \pt u_n + \Div \! f(x,u_n)
		-\eps_n \Delta u_n=\sigma(x,u_n)\dotW(t)
		+R(x,u_n),
	\end{split}	
\end{equation}
where $\eps_n\ton 0$. For the well-posedness of classical 
solutions to \eqref{eq:scl-viscous}, see \cite{Feng:2008ul} 
under the assumptions that $R\equiv 0$ and 
$f=f(u)$ does not depend on $x$.  For  
the general context provided by \eqref{eq:scl-viscous}, 
see \cite{Karlsen:2015ab} and 
\cite[Theorem 5.1]{Frid:2021us}. 

In the study of SPDEs on $\R^d$, weight functions are sometimes 
used. These weight functions are used to control 
the growth of solutions as they approach infinity, which 
in turn allows for the derivation of optimal 
conditions on the coefficients of the equations. 
The use of weighted $L^p$ spaces facilitates 
the analysis of stochastic conservation laws 
on $\R^d$ (see, e.g., \cite{Karlsen:2015ab}). 
Denote by $L^p(\chi dx)$ the 
weighted $L^p$ space of functions for which 
$$
\int_{\R^d} \abs{u(x)}^p\,\chi(x)\,dx<\infty,
$$
where $\chi$ is a weight function. The collection of 
relevant weights, denoted by $\cW$, consists 
of $\chi\in C^1(\R^d)\cap L^1(\R^d)$ for which $\chi(x)>0$ and 
$\abs{\nabla \chi(x)} \leq C_\chi \chi(x)$, 
for all $x\in \R^d$, where $C_\chi>0$ is a 
constant depending only on $\chi$.  
A simple example of a (smooth) weight function includes 
$\chi(x)=\chi_N(x)=(1+\abs{x}^2)^{-N}$, $N>d/2$. 
Any weight function $\chi\in \cW$ satisfies the properties
\begin{equation}\label{eq:weight-property}
	\abs{\chi(x+z)-\chi(x)}
	\lesssim \chi(x)\abs{z},
	\quad
	\sup_{\abs{x-y}\leq R}
	\frac{\chi(x)}{\chi(y)}\lesssim_R 1,
\end{equation}
which are used repeatedly in this paper. 
Clearly, $L^p(\R^d)\subset L^p(\chi dx)$, $p\in [1,\infty)$.
Moreover, $\chi^{-1}\in L^\infty_{\loc}(\R^d)$ implies that
$L^p(\chi dx)\subset L^p_{\loc}(\R^d)$.
Since $\chi \in L^1(\R^d)$, we also have 
$L^q(\chi dx) \subset L^p(\chi dx)$ 
for all $q,p$ such that $1 \leq p<q<\infty$. 

Now regarding a priori estimates for $u_n$, one can prove that 
there is an $n$-independent constant $C=C(\chi,p,r)$ such that 
\begin{equation}\label{eq:vv-basic-est}
	\Ex\norm{u_n}_{L^\infty(0,T;L^p(\chi dx))}^r\leq C, 
	\quad
	\Ex\abs{\int_{\R_+}\int_{\R} 
	\eps_n \abs{\nabla u_n}^2\chi(x)\dx\dt}^r\leq C,
\end{equation}
$\forall p,r\in [2,\infty)$ and for any $\chi\in \cW$, 
see \cite{Feng:2008ul}, \cite{Karlsen:2015ab}, 
and \cite[Remark 5.9]{Frid:2021us}. In the general case, $u_n$ does not 
exhibit $n$-independent $L^\infty$ 
and $BV$ estimates \cite[p.~711]{Chen:2012fk}.

In \cite{Chen:2012fk} (see also \cite{Debussche:2010fk}), 
the authors derived some basic quantitative compactness estimates. 
These $n$-uniform estimates, which were used and 
further refined in \cite{Karlsen:2015ab,Karlsen:2016aa} 
and \cite{ChenPang:2021}, take the form
\begin{equation}\label{eq:fracBVx-intro1}
	\Ex\int_{\R^d}\int_{\R^d} \bjd(z)
	\abs{u_n(t,x+z)-u_n(t,x-z)}\chi(x)\dx\dz
	\lesssim_T \delta^{\mu_x},
\end{equation}
for any $t\in (0,T)$ and some $\mu_x\in (0,1)$, 
where $\Seq{\bjd}_{\delta>0}$ is a mollifier sequence. 
One can prove that \eqref{eq:fracBVx-intro1} implies 
a ``fractional $BV$" estimate of the form
\begin{equation}\label{eq:fracBVx-intro2}
	\Ex \sup_{\abs{z}<\delta}
	\int_0^T\norm{u_n(t,\cdot+z)-u_n(t,\cdot)}_{L^1(\chi dx)}
	\dt\lesssim_T \delta^{\mu_x}, 
\end{equation}
for any $\delta>0$, see, 
e.g., \cite[Lemma 2]{Chen:2012fk} 
and Proposition \ref{prop:fracBVx} herein.  
Estimates like \eqref{eq:fracBVx-intro1} are 
often linked to the $L^1$ stability \textit{\`ala} Kru{\v{z}}kov 
of the solution operator.

Using the approximating SPDE and a 
modification \cite{Chen:2012fk,Karlsen:2016aa} 
of an interpolation technique due to Kru{\v{z}}kov, one 
can use the spatial estimate \eqref{eq:fracBVx-intro1} 
to establish that there exists $\mu_t\in (0,1)$ such that 
\begin{equation}\label{eq:fracBVt-intro1}
	\Ex \sup_{\tau\in (0,\delta)}
	\int_0^{T-\delta}\norm{u_n(t+\tau,\cdot)
	-u_n(t,\cdot)}_{L^1(\chi dx)}
	\dt\lesssim_T \delta^{\mu_t},
\end{equation}
see Proposition \ref{prop:fracBVt} for 
a general estimate of this type.

We refer to translation estimates like \eqref{eq:fracBVx-intro2} 
and \eqref{eq:fracBVt-intro1} as quantitative compactness estimates, 
see Section \ref{sec:quant-est} for further discussion and refinements.
They can be used to derive convergence results (via Cauchy sequence 
arguments) and error estimates for approximate solutions. 
Moreover, as part of the stochastic compactness method, one 
can use them to show that the laws $\cL(u_n)$ of $u_n$ 
form a tight sequence of probability  
measures on $L^1(\chi dx)$, which allows for the 
application of Skorokhod's representation theorem.

In \cite{Feng:2008ul}, the authors establish convergence 
of the viscosity approximations \eqref{eq:scl-viscous} 
using compensated compactness, assuming $d=1$, $R\equiv 0$, and the 
genuine nonlinearity of $f=f(u)$ ($f''\neq 0$ a.e.).
The main result of our paper is a 
refinement of the compensated compactness approach---in the 
spirit of \cite{Golse:2013aa}---that leads to 
a spatial compactness estimate 
like \eqref{eq:fracBVx-intro2} for the viscosity 
approximation, under a strengthened nonlinearity 
condition ($\abs{f''™}\geq c>0$). 
A temporal estimate \eqref{eq:fracBVt-intro1} 
follows from this estimate 
via Proposition \ref{prop:fracBVt}. 
Roughly speaking, in Section \ref{sec:CC}, we prove 
\eqref{eq:fracBVx-intro2} with $\delta_x=\frac14-\frac{1}{4p}$, 
for any finite $p>2$, assuming that the viscosity 
approximation $u_n$ is uniformly bounded 
in $L^p_{\omega,t,x}$ for any 
finite $p$, see \eqref{eq:vv-basic-est}. 

In the case that $u_n$ is bounded in $L^\infty_{\omega,t,x}$, 
we recover $\delta_x=\frac14$, which coincides with the 
known Besov regularity exponent (in $x$) of entropy solutions to 
conservation laws with one convex entropy 
and an entropy production that is a signed 
Radon measure \cite[Theorem 5]{Golse:2013aa}.
The quantitative version of compensated compactness  
allows for non-homogenous/discontinuous flux functions 
$f=f(x,u)$, in which case a signed measure arises naturally. 
The details will be presented elsewhere.

\medskip

For simplicity of presentation, we will 
in what follows assume that $W$ is 
a real-valued Wiener process and that 
$\sigma(x,u)$ is a real-valued function. 
The extension to a cylindrical Wiener 
process with corresponding 
operator-valued noise amplitude 
is standard, as discussed in \cite{DaPrato:2014aa} 
and the references cited earlier.

\section{Quantitative compactness estimates}
\label{sec:quant-est}

A subset $K$ of a metric space $(X,d)$ is 
precompact if its closure $\overline{K}$ is compact. 
A subset $K$ of a complete metric space $(X,d)$ is precompact 
if and only if it is totally bounded, meaning that 
for every $\eps>0$ there exists a finite cover of $K$ of open balls 
of radius $\eps$. We will use the  well-known Kolmogorov--Riesz--Fr\'echet
characterization of precompact subsets 
of $L^1_{\loc}(\R^d)$ in terms of the uniform 
continuity of the translation in $L^1(\R^d)$, 
see, e.g., \cite[Theorem 4.26]{Brezis:2010aa}. 

Using the fact that translation is continuous in $L^1$, we have 
the following simple but useful lemma. 

\begin{lemma}\label{lem:KRF-L1-compact}
Let $\cU\subset_b L^1(\chi dx)$, $\chi\in \cW$. 
Fix any $J\in L^1(\R^d)$ with 
$\supp (J) \subset B(0,R)$, 
$R>0$. Then $\cK:=J\star \cU=\set{J\star u: u\in \cU}$ 
is precompact in $L^1_{\loc}(\R^d)$.
\end{lemma}

\begin{proof}
Clearly, as $\chi>0$ on any set $D\subset\subset \R^d$, 
if $J\star \cU \chi:=\set{(J\star u) \chi: u\in \cU}$ 
is precompact in $L^1_{\loc}(\R^d)$, then 
$\cK=J\star \cU$ is precompact 
in $L^1_{\loc}(\R^d)$ as well.

Let us verify the precompactness of $J\star \cU \chi$ using 
the Kolmogorov--Riesz--Fr\'echet theorem. 
First, we claim that the set $J\star \cU \chi$ 
is bounded in $L^1(\R^d)$. Indeed,
\begin{align}\label{eq:J-u-chi-L1}
	&\norm{J\star u \chi}_{L^1(\R^d)}
	\leq \int_{\R^d}\int_{\R^d} \abs{J(y)}
	\abs{u(x-y)}\chi(x)\dx\dy
	\\ \notag & \quad
	=\int_{\R^d}\int_{\R^d} \abs{J(y)}
	\abs{u(x-y)}\chi(x-y)
	\frac{\chi(x)}{\chi(x-y)}\dx\dy
	\overset{\eqref{eq:weight-property}}{\lesssim_R}
	\norm{J}_{L^1(\R^d)}
	\norm{u}_{L^1(\chi dx)}.
\end{align}
Next, we verify the translation condition. 
For any translation $z\in \R^d$,
\begin{align*}
	&\norm{J\star u \chi)(\cdot+z)-J\star u\chi}_{L^1(\R^d)}
	\leq \int_{\R^d} 
	\abs{J\star u(x+z)-J\star u(x)}\chi(x)\dx\dy
	\\ & \qquad \qquad
	+\int_{\R^d} J\star u(x+z)
	\abs{\chi(x+z)-\chi(x)}\dx
	=: I_1+I_2,
\end{align*}
where $I_1 \overset{\eqref{eq:weight-property}}{\lesssim_R}
\norm{J(\cdot+z)-J}_{L^1(\R^d)}\norm{u}_{L^1(\chi dx)}$
and 
\begin{align*}
	I_2 & 
	\overset{\eqref{eq:weight-property}}{\leq}
	\abs{z}\int_{\R^d} \abs{J\star u(x+z)\chi(x+z)} 
	\frac{\chi(x)}{\chi(x+z)}\,dx
	\\ & 
	\overset{\eqref{eq:weight-property}}{\lesssim}
	\abs{z}\norm{J\star u \chi}_{L^1(\R^d)}
	\overset{\eqref{eq:J-u-chi-L1}}{\lesssim_R}
	\abs{z} \norm{J}_{L^1(\R^d)}
	\norm{u}_{L^1(\chi dx)}.
\end{align*}
Consequently, as $\abs{z}\to 0$,
$\norm{(J\star u \chi)(\cdot+z)-J\star u\chi}_{L^1(\R^d)}
\to 0$, uniformly in $J\star u \chi$ with $u\in \cU$.
An application of \cite[Theorem 4.26]{Brezis:2010aa} 
concludes the proof.
\end{proof}

\begin{definition}[approximate identity]
An approximate identity (or a convolution kernel) is 
a family $\seq{J_\delta}_{\delta>0}$ of 
$L^1(\R^d)$ functions $J_{\delta}:\R^d\to \R$ such that
\begin{align*}
	& \text{(i) $J_\delta$ is nonnegative and 
	$\int_{\R^d} J_\delta(x)\dx=1$, for each $\delta>0$},\\
	& \text{(ii) for each $h>0$, \; $\lim_{\delta\to 0}
	\int_{\abs{x}\ge h} J_\delta(x)\dx=0$.}
\end{align*}
\end{definition}

\begin{remark}
One simple (compactly supported) 
example is the standard (Friedrichs) mollifier. 
More generally, given a positive function $J$ with $\int J=1$, 
the rescaled family $J_\delta(x)=\frac{1}{\delta^d}J(x/\delta)$ 
supplies an approximate identify.
\end{remark}

Fix an approximate identity $\seq{\bjd}_{\delta>0}$. 
Consider a sequence $\seq{u_n}_{n\ge1}$ for which 
$\norm{u_n}_{L^1(\chi dx)}\lesssim 1$ and
\begin{equation}\label{eq:conv-conv-deltan}
	\lim_{\delta\to 0} \limsup_{n\to \infty} 
	\norm{u_n-\bjd\star u_n}_{L^1(\chi dx)}=0.
\end{equation}
Then $\seq{u_n}_{n\ge1}$ is precompact in $L^1_{\loc}(\R^d)$. 
Indeed, as $\seq{u_n}_{n\ge1}$ is bounded in $L^1(\chi dx)$, 
Lemma \ref{lem:KRF-L1-compact} supplies the 
precompactnes in $L^1(D)$ of the sequence $\seq{\bjd\star u_n}_{n\ge1}$, 
for each fixed $\delta>0$, and for any $D\subset \subset \R^d$.
Therefore, it is totally bounded in $L^1(D)$. In view 
of \eqref{eq:conv-conv-deltan} and since $\chi>0$ on $D$, 
the sequence $\seq{u_n}_{n\ge1}$ is 
also totally bounded---and thus precompact---in $L^1(D)$.

\begin{lemma}
Let $\seq{\bjd}_{\delta>0}$ be an approximate 
identity, and consider a sequence $\seq{u_n}_{n\in \N}$ 
of functions on $\R^d$. 
Suppose $\norm{u_n}_{L^1(\chi dx)}\lesssim 1$ and, 
for any $\delta>0$,
\begin{equation}\label{eq:compact-cond-kernel-space}
	\int_{\R^d}\int_{\R^d} \bjd(z)
	\abs{u_n(x+z)-u_n(x-z)}\chi(x)\dz\dx 
	\leq \rho(\delta),
\end{equation}
where $\rho:[0,\infty)\to [0,\infty)$ is an 
increasing function that is continuous at $0$ with $\rho(0)=0$ 
($\rho$ is independent of $n$). Then $\seq{u}_{n\in \N}$ is 
precompact in $L^1_{\loc}(\R^d)$.
\end{lemma}

\begin{proof}
Since
\begin{align*}
	&\int_{\R^d} \abs{u_n(x)-\bjd\star u_n(x)}\chi(x)\dx
	\\ & \quad 
	\leq \int_{\R^d}\int_{\R^d} 
	\bjd(x-y)\abs{u_n(x)-u_n(y)}\chi(x)\dy\dx
	\\ & \quad 
	= 2^d\int_{\R^d}\int_{\R^d}
	\bjdh(z)\abs{u_n(\tx+z)-u_n(\tx-z)}\chi(\tx+z) \dz\,\d\tx
	\\ & \quad
	\overset{\eqref{eq:weight-property}}{\lesssim}
	\int_{\R^d}\int_{\R^d}
	\bjdh(z)\abs{u_n(\tx+z)-u_n(\tx-z)}\chi(\tx)
	\dz\,\d\tx
	\overset{\eqref{eq:compact-cond-kernel-space}}{\leq}
	\rho\bigl(\delta/2\bigr)
	\leq \rho(\delta)
	\todelta 0,
\end{align*}
uniformly in $n$, the compactness 
condition \eqref{eq:conv-conv-deltan} follows.
\end{proof}

In what follows, we return to functions
$u_n=u_n(\omega,t,x):\Omega\times (0,T)\times \R^d$, 
$n\in \N$, depending also on the probability ($\omega$)
and temporal ($t$) variables.  The estimate 
\eqref{eq:compact-cond-kernel-space}, for 
suitable choices of  the ``modulus of continuity" $\rho(\cdot)$ 
and the approximate identity $\seq{\bjd}_{\delta>0}$, 
can be turned into a fractional $BV$ estimate like 
\eqref{eq:fracBVx-intro2}. This fact 
is related to known links between Sobolev, Besov, and Nikolskii 
fractional spaces (see, for example, \cite{Simon:1990fk}). 
More generally, we have

\begin{proposition}[quantitative compactness 
estimate in space]\label{prop:fracBVx}
Fix a weight $\chi\in\cW$ and a standard Friedrichs 
mollifier $\seq{\bjd}_{\delta>0}$. 
Consider a sequence $\seq{u_n}_{n\in \N}$ 
of functions satisfying 
$\Ex \int_0^T \int_{\R^d}\abs{u_n(t,x)}
\chi(x)\dx\dt\lesssim 1$ and
\begin{equation}\label{eq:fracBVx-ass1}
	\Ex \int_0^T\int_{\R^d}\int_{\R^d} \bjd(z)
	\abs{u_n(t,x+z)-u_n(t,x-z)}\chi(x)\dz\dx\dt 
	\leq \rho^x(\delta),
\end{equation}
where $\rho^x:[0,\infty)\to [0,\infty)$ is an 
increasing function that is continuous at $0$ with $\rho^x(0)=0$ 
($\rho^x$ is independent of $n$ but may depend on $\chi$, $T$).

Then $u_n$ satisfies the quantitative 
compactness (spatial translation) estimate
\begin{equation}\label{eq:un-spatial-L1}
	\Ex \sup_{\abs{z}<\delta} 
	\int_0^T\int_{\R^d}\abs{u_n(t,x+z)-u_n(t,x)}
	\, \chi(x)\dx\dt 
	\leq C	\rho^x(\delta), \quad \delta>0,
\end{equation}
where the constant $C=C(\chi,T)$ is independent of $n$.
\end{proposition}

\begin{proof}
For $\abs{z}>0$ and $\delta>0$, 
\begin{equation}\label{eq:BValpha-proof}
	\begin{split}
		& \int_0^T\int_{\R^d}\abs{u_n(t,x+z)-u_n(t,x)}
		\, \chi(x)\dx \dt
		\\ & \quad 
		\leq \int_0^T\int_{\R^d}\abs{\bjd\star u_n(t,x+z)
		-\bjd\star u_n(t,x)}\, \chi(x) \dx\dt
		\\ & \quad 
		\qquad +2 \int_0^T\int_{\R^d}
		\abs{\bjd\star u_n(t,x)-u_n(t,x)}
		\,\chi(x)\dx\dt=: A(z)+B,
	\end{split}
\end{equation}
where, by assumption,
\begin{align*}
	& B\leq 2\int_0^T\int_{\R^d}\int_{\R^d}
	\bjd(x-y)\abs{u_n(x)-u_n(y)}\chi(x)\dx \dy \dt
	\\ & \quad 
	= 2^{-d+1}\Ex \int_0^T\int_{\R^d}\int_{\R^d}
	\bjdh(z)\abs{u_n(\tx+z)-u_n(\tx-z)}\chi(\tx+z)
	\, d\tx \dz \dt
	\\ & \quad 
	\overset{\eqref{eq:weight-property}}{\lesssim}
	\Ex \int_0^T\int_{\R^d}\int_{\R^d}
	\bjdh(z)\abs{u_n(\tx+z)-u_n(\tx-z)}\chi(\tx)
	\,\d\tx \dz\dt 
	\overset{\eqref{eq:fracBVx-ass1}}{\lesssim} 
	\rho^x(\delta/2),
\end{align*}
so $B\lesssim\rho^x(\delta)$. For any $\delta>0$, we 
introduce the random variable
\begin{equation*}
	\sigma_u(\delta):=\sup_{\abs{z}<\delta}
	\int_0^T\int_{\R^d}
	\abs{u_n(t,x+z)-u_n(t,x)}\, \chi(x) \dx\dt,
\end{equation*}
which is a sub-additive/increasing 
modulus of continuity (for each fixed $\omega\in \Omega$). 
We obtain from \eqref{eq:BValpha-proof} that
\begin{equation}\label{eq:BValpha-proof-II}
	\Ex \sigma_u(\delta)
	\lesssim 
	\Ex \sup_{\abs{z}<\delta} A(z)
	+\rho^x(\delta).
\end{equation}

Note that
\begin{align*}
	& \bjd\star u_n(t,x+z)-\bjd\star u_n(t,x)
	\\ & \quad 
	=\int_{\R^d} \kappa_{z,\delta}(y)u_n(t,x-y)\dy
	=\int_{\R^d} \kappa_{z,\delta}(y)
	\bigl(u_n(t,x-y)-u_n(t,x)\bigr)\dy,
\end{align*}
with $\kappa_{z,\delta}(y)
:=\bjd(y+z)-\bjd(y)$ satisfying
$\int_{\R^d} \kappa_{z,\delta}(y)\dy=0$ 
and $\supp(\kappa_{z,\delta})
\subset B(0,\abs{z}+\delta)$.
For $\abs{z}\leq r\delta$, 
with $r\in (0,1)$ to be fixed later,
\begin{align*}
	A(z) &\leq \int_0^T\int_{\R^d} \int_{\R^d}
	\abs{\kappa_{z,\delta}(y)}
	\abs{u_n(t,x)-u_n(t,x-y)}\dy \chi(x) \dx
	\\ &  
	\leq \int_{\R^d}
	\abs{\kappa_{z,\delta}(y)}\dy\,
	\sigma_u\bigl(\abs{z}+\delta\bigr)
	\\ & 
	\leq C_1\abs{z}\norm{\Grad \bjd}_{L^\infty(\R^d)}
	\abs{B(0,\abs{z}+\delta)} 
	\sigma_u\bigl(\abs{z}+\delta\bigr)
	\\ &  
	\leq 
	C_2\frac{\abs{z}\left(\abs{z}+\delta\right)^d}{\delta^{d+1}}
	\sigma_u(\abs{z}+\delta)
	\leq C_3 r \sigma_u(r\delta+\delta)
	\leq C_4 r \sigma_u(r\delta),
\end{align*}
using also the sub-additivity of $\sigma_u(\cdot)$. 

Consequently, going back 
to \eqref{eq:BValpha-proof-II}, 
$\Ex \sigma_u(r\delta)\leq 
C_5r\Ex \sigma_u(r\delta)
+C_6\rho^x(r\delta)$. Fixing $r\in (0,1)$ such 
that $C_5r=\frac12$, we arrive at
$\Ex \sigma_u(r\delta)\leq 
\frac12\Ex \sigma_u(r\delta)
+C_6\rho^x(\delta)$ and thus $\Ex \sigma_u(\delta)
\leq C_7\rho^x(\delta)$. 
This concludes the proof of \eqref{eq:un-spatial-L1}.
\end{proof}

One can use the spatial estimate 
\eqref{eq:fracBVx-ass1} to derive a 
quantitative compactness estimate in time.   
To do that we need a version of a celebrated 
interpolation result due to Kru{\v{z}}kov \cite{Kruzkov:1969th}, 
which trades spatial regularity, here quantified in terms 
of \eqref{eq:fracBVx-ass1}, for temporal 
$L^1(\chi dx)$ continuity. 

\begin{proposition}[quantitative compactness estimate in time]
\label{prop:fracBVt}
Fix $m\in \N$ and a weight $\chi\in \cW\cap W^{m,\infty}(\R^d)$ 
such that $\abs{D^\alpha \chi (x)}\lesssim \chi(x)$ for 
any multi-index $\alpha$ with $\abs{\alpha}\leq m$. 
Let $\seq{\bjd}_{\delta>0}$ be 
an approximate identity such that the support of 
$\bjd$ is bounded independently of $\delta>0$ and 
$\norm{D^\alpha\bjd}_{L^1(\R^d)}
\lesssim \delta^{-\abs{\alpha}}$, for 
$\abs{\alpha}\leq m$, which includes, e.g., 
a standard Friedrichs mollifier. 
Consider a sequence $\seq{u_n}_{n\in \N}$ of  
functions on $\Omega\times (0,T)\times \R^d$, 
with $T>0$ fixed, satisfying 
\begin{align}
	& \Ex\int_0^T\int_{\R^d} \abs{u_n(t,x)}
	\chi(x)\dx\dt \lesssim 1,
	\label{eq:weighted-L1}
	\\ & 
	du_n=\sum_{\abs{\alpha}\leq m_F}D^\alpha F_{n}^{(\alpha)}\,dt
	+ \sum_{\abs{\alpha}\leq m_G}D^\alpha G_{n}^{(\alpha)}\,dW
	\quad \text{in $\Dp(\R^d)$, a.s.},
	\label{eq:interpol-weak-eqn}
	\\ & 
	\Ex\int_0^T \int_{\R^d}\abs{F_{n}^{(\alpha)}(t,x)}
	\chi(x)\dx\dt \lesssim 1, 
	\quad \forall \abs{\alpha}\le m_F, 
	\,\, m_F\leq m.
	\label{eq:Fan-bound}
	\\ & 
	\Ex \int_0^T \int_{\R^d}
	\abs{G_{n}^{(\alpha)}(t,x)}^2
	\chi(x)\dx\dt \lesssim 1, 
	\quad \forall \abs{\alpha}\le m_G, 
	\,\, m_G\leq m.
	\label{eq:Gan-bound}
\end{align}
Suppose the spatial compactness 
condition \eqref{eq:fracBVx-ass1} holds. 
Then, for any $\delta\in (0,T)$,
\begin{equation}\label{eq:un-temporal-L1}
	\Ex \sup_{\tau\in (0,\delta)}
	\int_0^{T-\delta}\int_{\R^d} 
	\abs{u_n(t+\tau,x)-u_n(t,x)}
	\chi(x)\dx\dt \leq \rho^t(\delta),
\end{equation}
where $\rho^t:[0,\infty)\to [0,\infty)$ is an 
increasing function that is continuous at $0$ 
with $\rho^t(0)=0$ ($\rho^t$ is independent of $n$ but may 
depend on $T,\chi$), see also \eqref{eq:rho-t-def}.
\end{proposition}

\begin{proof}
Recall that $u_n$ satisfies the $L^1$ bound \eqref{eq:weighted-L1}.
For $\nu>0$, set
\begin{equation}\label{eq:vdelta-def}
	v_{n,\nu}(t,x)=\int_{\R^d}\frac{1}{2^d}
	\bjn\left(\frac{x-y}{2}\right)
	\signb{d_n(t,y)}\dy,
\end{equation}
where $d_n(t,x):=u_n(t+\tau,x)-u_n(t,x)$, for 
$\tau\in (0,\delta)$, $\delta\in (0,T)$.
We have
\begin{equation}\label{eq:v-der-estimate}
	\norm{D^{\alpha}v_{n,\nu}(t,\cdot)}_{L^\infty(\R^d)}
	\lesssim 1/\nu^{\abs{\alpha}}, 
	\quad \abs{\alpha}\leq m,
\end{equation}
where the right-hand side is independent of $n,t$. 
In other words, we have a.s.~that 
$v_{n,\nu}\in L^\infty(0,T;W^{m,\infty}(\R^d))$.

Note that, by $\absb{\abs{a}-a\sgn(b)}
\leq 2 \abs{a-b}$ $\forall a,b\in \R$,
\begin{align*}
	& \absb{\abs{d_n(t,x)}-d_n(t,x)v_{n,\nu}(t,x)}
	\leq \frac{1}{2^{d-1}}
	\int_{\R^d} \bjn\left(\frac{x-y}{2}\right)
	\abs{d_n(t,x)-d_n(t,y)}\dy,
\end{align*}
As a result,
\begin{align*}
	&\int_0^{T-\tau} \int_{\R^d} 
	\absb{\abs{d_n(t,x)}-d_n(t,x)v_{n,\nu}(t,x)} \chi(x) \dx\dt
	\\ & \quad \leq 
	\frac{1}{2^{d-1}} \int_0^{T-\tau} \int_{\R^d}\int_{\R^d}
	\bjn\left(\frac{x-y}{2}\right) 
	\abs{d_n(t,x)-d_n(t,y)}\chi(x)\dx\dy\dt
	\\ & \quad = 
	2\int_0^{T-\tau} \int_{\R^d}\int_{\R^d}
	\bjn(z) \abs{d_n(t,\tx+z)-d_n(t,\tx-z)}
	\chi(\tx+z)\dtx\dz\dt.
\end{align*}
Utilising \eqref{eq:weight-property} 
(and the bounded support of $\bjn$),
$\bjn(z) \chi(\tx+z)\lesssim \bjn(z) \chi(\tx)$. 
Hence
\begin{equation}\label{eq:d-abs-approx}
	\Ex\int_0^{T-\tau} \int_{\R^d} 
	\absb{\abs{d_n(t,x)}
	-d_n(t,x)v_{n,\nu}(t,x)} \chi(x) \dx\dt
	\overset{\eqref{eq:fracBVx-ass1}}{\lesssim} \rho^x(\nu).
\end{equation}

The weak form of the SPDE \eqref{eq:interpol-weak-eqn} implies
\begin{align}
	& \notag 
	\abs{\int_{\R^d} 
	d_n(t,x)v(t,x)\,\chi(x)\dx}
	\leq \sum_{\abs{\alpha}\le m_F} \abs{\int_t^{t+\tau} \int_{\R^d}
	F_{n}^{(\alpha)}(s,x) \cdot D^\alpha 
	\bigl(v(t,x)\chi(x)\bigr) \dx\ds}
	\\ & \qquad\qquad
	+ \sum_{\abs{\alpha}\le m_G}
	\abs{\int_{\R^d}\left(\int_t^{t+\tau} 
	G_{n}^{(\alpha)}(s,x)\, dW(s)\right)
	D^{\alpha}\bigl(v(t,x) \, \chi(x)\bigr)\dx},
	\label{eq:weak-form-SPDE}
\end{align}
for all $v \in L^\infty(0,T;W^{m,\infty}(\R^d))$. 
Combining \eqref{eq:weak-form-SPDE} 
and \eqref{eq:d-abs-approx} yields
\begin{align*}
	I & :=\Ex\int_0^{T-\delta}
	\sup_{\tau\in (0,\delta)}\int_{\R^d} 
	\abs{d_n(t,x)}\, \chi(x)\dx\dt
	\leq C_1\rho^x(\nu)+A(\delta,\nu)
	+B(\delta,\nu),
\end{align*}
where
\begin{align*}
	& A(\delta,\nu)=\sum_{\abs{\alpha}\le m_F}  
	\int_0^{T-\delta} 
	\Ex\int_t^{t+\delta} \int_{\R^d}
	\abs{F_{n}^{(\alpha)}(s,x)} 
	\abs{D^\alpha \bigl(v_{n,\nu}(t,x)\chi(x)\bigr)}
	\dx\ds\dt,
	\\ &
	B(\delta,\nu)= \sum_{\abs{\alpha}\le m_G}
	\int_0^{T-\delta}
	\Ex \sup_{\tau\in (0,\delta)} 
	\abs{M^{(\alpha)}(\tau;t)}\dt,
	\\ & 
	M^{(\alpha)}(\tau;t)=\int_{\R^d}
	\left(\int_t^{t+\tau} 
	G_{n}^{(\alpha)}(s,x) \dW(s)\right)
	D^\alpha\bigl(v_{n,\nu}(t,x)\,\chi(x)\bigr) \dx,
\end{align*}
and $v_{n,\nu}$ is defined in \eqref{eq:vdelta-def}. 
By assumption, the weight function $\chi$ belongs to $W^{m,\infty}$ 
and satisfies $\abs{D^\alpha \chi (x)}\lesssim \chi(x)$ for 
$\abs{\alpha}\leq m$. Thus, by \eqref{eq:v-der-estimate},
\begin{equation}\label{eq:Dalpha-v-chi}
	\abs{D^\alpha \bigl(v_{n,\nu}(t,x)\chi(x)\bigr)}
	\lesssim \frac{\chi(x)}{\nu^{\abs{\alpha}}},
\end{equation}
and so
\begin{align*}
	& I\leq C_1\rho^x(\nu)
	+\frac{\delta}{\nu^{m_F}}\sum_{\abs{\alpha}\le m_F} 
	\Ex\int_0^T \int_{\R^d}
	\abs{F_{n}^{(\alpha)}(t,x)}\chi(x)\dx\dt
	+B(\delta,\nu)
	\\ & \quad 
	\overset{\eqref{eq:Fan-bound}}{\leq} 
	C_1\rho^x(\nu)
	+C_2\frac{\delta}{\nu^{m_F}}+B(\delta,\nu), 
\end{align*}
for some $(n,\delta,\nu)$-independent constants $C_1,C_2$. 

Next, using first \eqref{eq:Dalpha-v-chi} and 
then the Cauchy--Schwarz inequality, it follows that
$\abs{M^{(\alpha)}(\tau;t)} \lesssim_\chi 
\frac{1}{\nu^{m_G}}\norm{\int_t^{t+\tau} 
G_{n}^{(\alpha)}(s) \dW(s)}_{L^2(\chi dx)}$. 
By the (Hilbert-space valued) BDG 
inequality \cite[page 174]{DaPrato:2014aa},
and again the Cauchy--Schwarz inequality,
\begin{align*}
	& B(\delta,\nu)
	\lesssim_{\chi,T} \frac{1}{\nu^{m_G}}\!\!
	\sum_{\abs{\alpha}\le m_G}\!\!
	\left(\Ex \int_0^{T} \int_t^{t+\delta} 
	\int_{\R^d}\abs{G_{n}^{(\alpha)}(s,x)}^2 
	\,\chi(x) \dx \ds \dt \right)^{1/2} 
	\overset{\eqref{eq:Gan-bound}}{\lesssim}
	\frac{\delta^{1/2}}{\nu^{m_G}},
\end{align*}
so that $I\leq C_1\rho^x(\nu)
+C_2\frac{\delta}{\nu^{m_F}}+C_3 \frac{\delta^{1/2}}{\nu^{m_G}}$. 
The claim \eqref{eq:un-temporal-L1} follows by setting 
\begin{equation}\label{eq:rho-t-def}
	\rho^t(\delta)=\inf\limits_{\nu>0} 
	\left(C_1\rho^x(\nu)+C_2\frac{\delta}{\nu^{m_F}}
	+C_3 \frac{\delta^{1/2}}{\nu^{m_G}}\right).
\end{equation} 
\end{proof}

\section{Quantitative compensated compactness}
\label{sec:CC}

Consider the viscosity approximation $u_n$ of \eqref{eq:scl} 
with $d=1$, $R\equiv 0$, and $f=f(u)$. 
According to \cite{Feng:2008ul}, see also 
\cite{Karlsen:2015ab} and \cite[Remark 5.9]{Frid:2021us}, 
there exists a unique solution $\ue$---continuous in $t$ and 
smooth in $x$---of the SPDEs
\begin{align}
	& d\ue +\px f(\ue)\, dt
	=\eps \pxx \ue\,dt +\sigma(x,\ue)\,dW, 
	\label{eq:VV-1D} \\ & 
	d\eta(\ue)+\px q(\ue)\,dt
	=\eps\pxx \eta(\ue)\,dt-\mu_\eps
	\label{eq:VV-entropy-ineq}
	\\ & \qquad \qquad
	+\frac12 \eta''(\ue)\sigma^2(x,\ue)\,dt 
	+\eta'(\ue)\sigma(x,\ue)\,dW, 
	\notag
\end{align}
where $\eta\in C^2(\R)$, $q'=\eta'f'$, and 
$\mu_\eps := \eta''(\ue) \eps \abs{\px\ue}^2$. 
Moreover, $\forall p,r\in [2,\infty)$ and 
for any weight $\chi\in \cW$,
\begin{equation}\label{eq:stoch-ue-est1}
	\begin{split}
		& \Ex\norm{\ue}_{L^\infty(0,T;L^p(\chi dx))}^r
		\leq C_\chi, \quad
		\Ex\abs{\int_0^T\int_{\R} 
		\eps \abs{\px\ue}^2 \chi(x)\dx\dt}^r \leq C_\chi,
		\\ & \quad \text{and} \quad
		\Ex\abs{\int_0^T\int_{\R} 
		\chi\, d\abs{\mu_\eps}}^r
		\leq C_\chi.
	\end{split}
\end{equation}
The entropy balance \eqref{eq:VV-entropy-ineq} follows 
from the viscous SPDE \eqref{eq:VV-1D} and 
the spatial and temporal (It\^{o}) chain rules. 

The flux $f$ and entropy $\eta$ 
are assumed to satisfy the nonlinearity assumptions of 
the next lemma, which is a global version of 
\cite[Lemma 5.2]{Golse:2013aa} that 
will be utilised later. We refer to \cite{Golse:2013aa} 
for the proof.

\begin{lemma}\label{lem:nonlinear-flux-est-global}
Suppose $f,\eta$ are $C^1$ functions
such that
\begin{align}
	\label{eq:f-nonlin-ass}
	& f'(v)-f'(w)\ge C_f\left(v-w\right)^{p_f},
	\quad -\infty\le w<v\le \infty, \,\, 
	p_f\ge 1,
	\\ &
	\label{eq:eta-nonlin-ass}
	\eta'(v)-\eta'(w)\ge C_\eta\left(v-w\right)^{p_\eta},
	\quad -\infty\le w<v\le \infty, \,\, 
	p_\eta\ge 1,
\end{align}
for some constants $C_f, C_\eta>0$. 
Then, for all $v,w\in \R$,
$$
(w-v)\bigl(q(w)-q(v)\bigr)
-\bigl(\eta(w)-\eta(v)\bigr)\bigl(f(w)-f(v) \bigr) 
\geq C_{f,\eta}\abs{w-v}^{p_f+p_\eta+2},
$$
where $C_{f,\eta}=\frac{C_f C_\eta}
{\left(1+p_f+p_\eta\right)
\left(2+p_f+p_\eta\right)}$.
\end{lemma}

\begin{remark}\label{rem:f-convex}
Suppose $f\in W^{2,\infty}(\R)$ is such that 
$\abs{f''(\xi)}\ge c>0$ for a.e.~$\xi\in\R$. 
This is a ``quantitive" version of the classical 
nonlinearity condition $f''(\xi)\neq 0$ 
for a.e.~$\xi\in \R$ \cite{Lu:2003lr,Tartar:1983ul}.  
Then, choosing $\eta=f$ as an entropy flux, it 
follows that $(w-v)\bigl(q(w)-q(v)\bigr)
-\bigl(f(w)-f(v) \bigr)^2\geq \frac{c^2}{8}\abs{w-v}^{4}$.
\end{remark}

\begin{remark}
Compared to \cite[Lemma 5.2]{Golse:2013aa}, 
the assumptions in Lemma \ref{lem:nonlinear-flux-est-global}
are global, since we do not know that $\ue$ is 
bounded in $L^\infty_{\omega,t,x}$. 
In what follows, we always assume 
$\eta\in W^{2,\infty}_{\operatorname{loc}}(\R)$ is 
such that $\eta,\eta',\eta''$ are at most polynominally 
growing (including $\eta=f)$: for some $p_0\ge 2$,
\begin{equation}\label{eq:eta-pol}
	\abs{\eta(u)}\lesssim 1+\abs{u}^{p_0}, \quad 
	\abs{\eta'(u)}\lesssim 1+\abs{u}^{p_0-1}, \quad
	\abs{\eta''(u)}\lesssim 1+\abs{u}^{p_0-2}.
\end{equation} 
To ensure the validity of the a priori 
estimates in \eqref{eq:stoch-ue-est1}, 
we assume that 
\begin{equation}\label{eq:sigma-lin}
	\abs{\sigma(x,u)}\lesssim 1+\abs{u}.
\end{equation}
\end{remark}

The main result is the following spatial 
compensated compactness estimate:

\begin{theorem}[quantitative compactness estimate, 
viscosity approximation]
\label{thm:quant-CC-est}
Fix a weight $\chi\in \cW$. Let $u_n$ be a classical 
solution to the viscous SPDEs 
\eqref{eq:VV-1D}, \eqref{eq:VV-entropy-ineq} with 
$\eps=\eps_n\to 0$ as $n\to \infty$. Suppose 
\eqref{eq:f-nonlin-ass}, \eqref{eq:eta-nonlin-ass}, 
\eqref{eq:eta-pol}, and \eqref{eq:sigma-lin} hold.
Set $\mu:=1-\frac{1}{p}$, $p\in [2,\infty)$. 
Then, for any $z\in \R$ with $\abs{z}<1$,
\begin{equation}\label{eq:stoch-CC-est}
	\Ex \int_0^T \int_{\R} 
	\abs{u_n(t,x+z)-u_n(t,x)}^{p_f+p_\eta+2}
	\,\chi(x)\dx\dt\lesssim_{\chi,p}\abs{z}^{\mu}.
\end{equation}
\end{theorem}

\begin{remark}
By H\"older's inequality, it follows 
from \eqref{eq:stoch-CC-est} that
\begin{equation}\label{eq:spatial-ass-standard}
	\Ex\int_0^T\int_{\R}
	\abs{u_n(t,x+z)-u_n(t,x)}\,\chi(x)\dx\dt
	\leq C\abs{z}^{\mu_x},
\end{equation}
where $\mu_x:=\frac{\mu}{p_f+p_\eta+2}$ and $C=C(T,\chi,p)$. 
Note carefully that this estimate appears to 
be slightly weaker than the spatial compactness 
estimate \eqref{eq:fracBVx-intro2} as there is 
a \, missing $\sup_{\abs{z}<\delta}$ inside the expectation operator. 
It is not immediately clear how to recover this supremum; 
the available maximal (martingale) inequalities apply 
only to the temporal variable $t$ and not an arbitrary parameter $z$. 
However, we can use Proposition \ref{prop:fracBVx} 
to recover the supremum by considering a standard 
Friedrichs mollifier $\seq{\bjd}_{\delta>0}$.
Observe that \eqref{eq:spatial-ass-standard} implies
\begin{equation}\label{eq:CC-estimate-tmp1}
	\begin{split}
		&\Ex\int_0^T\int_{\R}\int_{\R}
		\bjd(z)\abs{u_n(t,x+z)-u_n(t,x-z)}\,\chi(x)\dz\dx\dt
		\\ & \qquad 
		\leq 2C\int_{\R}\bjd(z)\abs{z}^{\mu_x}\, dz
		\leq 2C\delta^{\mu_x}.
	\end{split}
\end{equation}
Given \eqref{eq:spatial-ass-standard}, 
Proposition \ref{prop:fracBVx}---via the estimate 
\eqref{eq:CC-estimate-tmp1}---supplies 
the spatial estimate \eqref{eq:fracBVx-intro2} 
with $\sup_{\abs{z}<\delta}$ inside the 
expectation operator. Besides, by Proposition \ref{prop:fracBVt}, 
\eqref{eq:CC-estimate-tmp1} also implies 
the temporal compactness estimate \eqref{eq:fracBVt-intro1}. 

It is worth noting that these spatial and 
temporal estimates do not require the measure $\mu_\eps$ in 
the entropy balance equation \eqref{eq:VV-entropy-ineq} 
to be positive. This means that they can be 
used to analyze stochastic conservation laws 
with discontinuous ($BV$) flux, which 
will be discussed further in future work.
\end{remark}

\begin{remark}
Under the assumption of 
Remark \ref{rem:f-convex}, we have $p_f+p_\eta+2=4$; 
thus the $L^1_x$ compactness 
estimate \eqref{eq:fracBVx-intro2} holds with 
$\rho^x(\delta)=\delta^{\mu_x}$ and $\mu_x:=\frac14-\frac{1}{4p}$, 
for any $p\in [2,\infty)$. If $\norm{\ue}_{L^\infty_{\omega,t,x}}
\lesssim 1$, then we can improve 
this to $\mu_x=1/4$, which is 
consistent with \cite{Golse:2013aa}.
\end{remark}

\subsection{Stochastic interaction lemma}

Consider the two It\^{o} SPDEs
\begin{equation}\label{eq:SPDE-system}
	\begin{split}
		& dA + \px B\dt = C_A\dt + \sigma_A\dW,
		\\ &
		dD + \px E\dt = C_D\dt + \sigma_D\,dW,
	\end{split}
\end{equation}
which hold weakly in $x$, almost surely. 
Here $A=A(\omega,t,x),D=D (\omega,t,x)$ 
belong to $C([0,T];L^1(\R^d))$ a.s., satisfy
$\Ex \norm{\bigl(A,D\bigr)(t)}_{L^1(\R)}^2<\infty$ 
$\forall t\in [0,T]$, and $A,D\to 0$ as $\abs{x}\to 0$, for 
each fixed $(\omega,t)$. Besides, $\bigl(B,C_A,E,C_D\bigr)=
\bigl(B,C_A,E,C_D\bigr)(\omega,t,x)$ 
satisfy $\Ex \int_0^T \norm{\bigl(B,C_A,E,C_D\bigr)(t)}_{L^1(\R)}
\dx\dt<\infty$. The noise amplitudes $\bigl(\sigma_A,\sigma_D\bigr)
=\bigl(\sigma_A,\sigma_D\bigr)(\omega,t,x)$ satisfy $\Ex \int_0^T 
\norm{\bigl(\sigma_A(t),\sigma_D(t)\bigr)}_{L^p(\R)}^2\dt<\infty$ 
for $p=1,2$. 

The processes $A,D,\sigma_A,\sigma_D$ and $W$ live 
on the stochastic basis introduced in Section \ref{sec:intro}, 
and they are assumed to be appropriately measurable with 
respect to the filtration, so all 
relevant stochastic integrals are well-defined 
in the sense of It\^{o}.

In what follows, we need the spatial 
anti-derivatives of $A$ and $D$,
$$
\cA(t,y):=\int_{-\infty}^y A(t,y)\, dx,
\quad 
\cD(t,x):=\int_x^{\infty} D(t,y)\, dy,
$$
as well the spatial anti-derivatives 
of $\sigma_A$ and $\sigma_D$,
$$
\Sigma_A(t,y):=\int_{-\infty}^y \sigma_A(t,x)\dx,
\quad 
\Sigma_D(t,x):=\int_x^{\infty} \sigma_D(t,y)\dy.
$$

The following lemma is a stochastic adaptation of 
a result from Golse and Perthame \cite[Section 2]{Golse:2013aa}. 

\begin{lemma}[stochastic interaction identity]
\label{lem:stoch-interact}
For $t\in [0,T]$, define
$$
I(t)=\iint\limits_{x<y} A(t,x)D(t,y)\dx\dy.
$$
Then the following identity holds a.s.:
\begin{equation}\label{eq:x-interaction-stoch}
	\begin{split}
		&\int_0^T\int_{\R}
		\bigl(AE-DB \bigr)(t,x)\dx\dt
		\\ & \quad 
		=-\int_0^T\int_{\R} C_A(t,x)\cD(t,x)\dx\dt
		-\int_0^T\int_{\R} \cA(t,y)C_D(t,y)\dy\dt
		\\ & \quad \qquad
		-\int_0^T\int_{\R}
		\sigma_A(t,x)\cD(t,x)\dx\dW
		-\int_0^T\int_{\R} \cA(t,y)\sigma_D(t,y)\dy\dW
		\\ & \quad\qquad
		-\int_0^T I_\sigma(t)\,dt+I(T)-I(0),
	\end{split}
\end{equation}
where 
$I_\sigma(t):=\iint\limits_{x<y}
\sigma_A(t,x)\sigma_D(t,y)\dx\dy$ 
satisfies
\begin{equation}\label{eq:noise-interact-new}
	I_\sigma(t)=\int_{\R}\sigma_A(t,x)\Sigma_D(t,x)\dx
	=\int_{\R}\Sigma_A(t,y)\sigma_D(t,y)\dy.
\end{equation}
\end{lemma}

\begin{proof}
Modulo an approximation argument involving spatial 
mollification, we may assume that the SPDE system 
\eqref{eq:SPDE-system} holds pointwise in $x$. Thus,
using the real-valued It\^{o} product formula, we 
compute as follows:
\begin{align*}
	& dI(t) 
	= \iint\limits_{x<y} 
	\Bigl(-\px B(t,x)\,dt+C_A(t,x)\,dt+\sigma_a(t,x)\dW
	\Bigr)D(t,y)\dx\dy
	\\ & \qquad \qquad 
	+\iint\limits_{x<y} A(t,x)
	\Bigl(-\py E(t,y)\,dt+C_D(t,y)\,dt+\sigma_D(t,y)\dW
	\Bigr)\dx \dy+I_\sigma(t)
 	\\ & \qquad =
	\left(\,\,\, \iint\limits_{x<y}C_A(t,x)D(t,y)
	\dx\dy\right)\dt
	+\left(\,\,\,\iint\limits_{x<y}\sigma_A(t,x)D(t,y)
	\dx\dy\right)\dW
	\\ & \qquad \qquad 
	+\left(\,\,\,\iint\limits_{x<y}A(t,x)C_D(t,y)
	\dx\dy\right)\dt
	+\left(\,\,\,\iint\limits_{x<y}A(t,x)\sigma_D(t,y)
	\dx\dy\dt\right)\dW
	\\ & \qquad \qquad
	+\left(\,\,\,\int_{\R} -B(t,y)D(t,y)\,dy \right)\,dt
 	+\left(\,\,\, \int_{\R} A(t,x)E(t,x)\,dx\right)\,dt
 	+I_\sigma(t),
\end{align*}
so that
\begin{align*}
	dI(t) & =\left(\,\,\int_{\R}
	\bigl(AE-DB\bigr)(t,x)\dx\right)\dt+I_\sigma(t)
	\\ & \quad 
	+\left(\,\,\, \iint\limits_{x<y}C_A(t,x)D(t,y)
	\dx\dy\right)\dt
	+ \left(\,\,\, \iint\limits_{x<y}A(t,x)C_D(t,y)
	\dx\dy\right)\dt
	\\ & \quad 
	+\left(\,\, \, \iint\limits_{x<y}\sigma_A(t,x)D(t,y)
	\dx\dy\right)\dW
	+ \left(\,\, \, \iint\limits_{x<y} 
	A(t,x)\sigma_D(t,y) \dx\dy\right)\dW.
\end{align*}
Observe that 
$$\iint\limits_{x<y} C_A(t,x)D(t,y)\dx\dy
=\int_{\R} C_A(t,x)\cD(t,x)\dx
$$ 
and 
$$
\iint\limits_{x<y} A(t,x)C_D(t,y)\dx\dy
=\int_{\R}\cA(t,y)C_D(t,y)\dy.
$$ 
Similarly, we have
$$
\left(\,\, \, \iint\limits_{x<y}\sigma_A(t,x)D(t,y)
\,dx\,dy\right)\dW=\left( \,\int_{R}
\sigma_A(t,x)\cD(t,x)\dx\right)\dW,
$$ 
and 
$$
\left(\,\, \, \iint\limits_{x<y} 
A(t,x)\sigma_D(t,y) \,dx\,dy\right)\,dW
=\left(\,\, \, \int_{\R} 
\cA(t,y)\sigma_D(t,y)\,dy\right)\,dW,
$$ 
and $I_\sigma(t)$ satisfies \eqref{eq:noise-interact-new}.  
Finally, upon integrating in time, 
we obtain \eqref{eq:x-interaction-stoch}.
\end{proof}

\subsection{Proof of Theorem \ref{thm:quant-CC-est}}
In what follows, we will write $\ue$ instead of $u_n$ 
for the solution of \eqref{eq:VV-1D}, where $\eps = \eps_n \to 0$ 
as $n \to \infty$. We will use $h$ to denote the 
spatial translation step instead of $z$, and we will 
define the spatial difference operator of a function 
$F:\mathbb{R} \to \mathbb{R}$ with step size $h \in (-1,1)$ 
as $\Delta_h F(x)=F(x+h)-F(x)$. 

Applying $\Delta_h$ 
to \eqref{eq:VV-1D}, \eqref{eq:VV-entropy-ineq} gives
\begin{align*}
	& d \bigl(\chi\Delta_h \ue\bigr) 
	+\px \bigl(\chi \Delta_h f(\ue)\bigr)\, dt
	\\ & \quad\quad
	=\px \chi \Delta_h f(\ue)\, dt
	+\eps \chi \pxx \Delta_h\ue\,dt
	+\chi \Delta_h\sigma(x,u)\,dW, 
	\\ & 
	d \bigl(\chi \Delta_h\eta(\ue)\bigr)
	+\px \bigl(\chi \Delta_h q(\ue)\bigr)\,dt
	\\ & \quad\quad
	=\px \chi \Delta_h q(\ue)\,dt
	+\eps\chi \pxx \Delta_h\eta(\ue)\,dt
	-\chi\Delta_h\mu_\eps
	\\ & \quad\quad\quad
	+\frac12 \chi \Delta_h 
	\left(\eta''(\ue)\sigma^2(x,\ue)\right)\,dt 
	+ \chi \Delta_h \bigl(\eta'(\ue) 
	\sigma(x,\ue)\bigr)\,dW. 
\end{align*}

Setting
\begin{align*}
	& A=\chi\Delta_h \ue, \quad 
	B=\chi \Delta_h f(\ue), \quad
	\\ & 
	C_A = \px \chi \Delta_h f(\ue)
	+\chi\eps \pxx \Delta_h\ue,
	\quad 
	\sigma_A=\chi\Delta_h\sigma(x,\ue)
\end{align*}
and
\begin{align*}
	& D=\chi \Delta_h \eta(\ue), 
	\quad 
	E=\chi \Delta_h q(\ue),
	\quad 
	C_D=C_{D,0}+
	\chi\bigl( C_{D,1}+C_{D,2}+C_{D,3}\bigr),  
	\\ & 
	C_{D,0}=
	\px \chi\Delta_h q(\ue),
	\quad
	C_{D,1}=\eps\pxx \Delta_h\eta(\ue),
	\quad 
	C_{D,2}=-\Delta_h\mu_\eps,
	\\ &
	C_{D,3}= \frac12 \Delta_h 
	\left(\eta''(\ue)\sigma^2(x,\ue)\right),
	\quad 
	\sigma_D=\chi \Delta_h 
	\bigl(\eta'(\ue) \sigma(x,\ue)\bigr),
\end{align*}
the SPDE system \eqref{eq:SPDE-system} holds. 

If $\chi$ is a weight function, then it follows that 
both $\sqrt{\chi}$ and $\chi^2$ are also weight functions. 
In other words, $\chi\in \cW \Longrightarrow 
\sqrt{\chi},\chi^2\in \cW$.

Given the stochastic interaction 
estimate \eqref{eq:x-interaction-stoch} 
and Lemma \ref{lem:nonlinear-flux-est-global}, we obtain
\begin{equation}\label{eq:stoch-Delta-x}
	\begin{split}
		&\Ex\iint
		\abs{\Delta_h \ue(t,x)}^{p_f+p_{\eta}+2}
		\,\chi^2(x)\dx\dt\lesssim
		\sum_{i=1}^8 J_i,
	\end{split}
\end{equation}
where
\begin{align*}
	& J_1 = -\Ex\iint \px \chi \Delta_h f(\ue) 
	\left(\int_x^\infty \chi \Delta_h \eta(\ue)\dy\right)\dx\dt,
	\\ & 
	J_2 = -\Ex\iint \chi \eps \pxx \Delta_h\ue 
	\left(\int_x^\infty \chi \Delta_h \eta(\ue)\dy\right)\dx\dt,
	\\ &
	J_3= -\Ex\iint \bigl(\px \chi\Delta_h q(\ue)\bigr(t,y) 
	\left(\int_{-\infty}^y \chi \Delta_h \ue\,dx\right)\dy\dt,
	\\ &
	J_4= -\Ex\iint \bigl(\chi \eps 
	\pxx \Delta_h\eta(\ue)\bigr)(t,y) 
	\left(\int_{-\infty}^y \chi \Delta_h \ue\,dx\right)\dy\dt,
	\\ &
	J_5=\Ex\iint \bigl(\chi\Delta_h \mu_\eps\bigr)(t,y)
	\left(\int_{-\infty}^y \chi \Delta_h \ue\dx\right)\dy\dt,
	\\ & 
	J_6 = -\frac12 \Ex\iint  \bigl(\chi \Delta_h
	\left(\eta''(\ue)\sigma^2(y,\ue)\right)\bigr)(t,y)
	\left(\int_{-\infty}^y \chi \Delta_h 
	\ue\dx\right)\dy\dt,
\end{align*}
and 
\begin{align*}
	J_7 & = -\Ex\int_0^T 
	\Biggl(\,\,\,\iint\limits_{x<y} 
	\chi(t,x)\Delta_h\sigma(x,\ue(t,x)))
	\\ & \qquad\qquad\qquad\quad\times 
	\chi(t,y)\Delta_h 
	\bigl(\eta'(\ue(t,y)) \sigma(y,\ue(t,y))
	\bigr)\,dx\,dy \Biggr)\,dt.
\end{align*}
Finally, 
$J_8=\Ex\iint\limits_{x<y} \bigl(\chi\Delta_h \ue\bigr)(t,x)
\bigl(\chi \Delta_h \eta(\ue)\bigr)(t,y)\dx \dy
\Big|_{t=0}^{t=T}$.

\smallskip

\noindent\underline{Estimates of $J_1$ and $J_3$.} 
First, 
\begin{align*}
	\mathcal{J}_1 & :=\int_x^\infty \chi(y)
	\Delta_h \eta(\ue(t,y))\dy
	= \int_x^\infty \Delta_h\bigl(
	\chi(y)) \eta(\ue(t,y))\bigr)\dy
	\\ & \qquad 
	-\int_x^\infty \Delta_h\chi(y)
	\eta(\ue(t,y+h))\dy
	=:\mathcal{J}_{1,1}+\mathcal{J}_{1,2},
\end{align*}
where
\begin{align*}
	\abs{\mathcal{J}_{1,1}}&=\abs{\int_x^{x+\abs{h}}
	\chi(y)\eta(\ue(t,y))\,dy} 
	\leq \abs{h}^{\frac{p-1}{p}}
	\norm{\eta(\ue)}_{L^\infty(0,T;L^p(\chi dx))}
\end{align*}
and, using $\abs{\px \chi(x)}\lesssim \chi(x)$ 
and \eqref{eq:weight-property}, 
\begin{align*}
	\abs{\mathcal{J}_{1,2}}
	& =\abs{\int_x^\infty \int_0^h 
	\partial_{y'}\chi(y+y')\dy'
	\eta(\ue(t,y+h))\dy}
	\lesssim \abs{h}\norm{\eta(\ue)}_{L^\infty(0,T;L^1(\chi dx))}.
\end{align*}
As a result,
\begin{equation}\label{eq:anti-der-chi-eta}
	\begin{split}
		\abs{\mathcal{J}_1}
		\lesssim \abs{h}^{1-\frac{1}{p}}
		\norm{\eta(\ue)}_{L^\infty(0,T;L^p(\chi dx))},
	\end{split}
\end{equation}
and thus, by Young's product inequality,
\begin{align*}
	\abs{J_1} 
	\lesssim \abs{h}^{1-\frac{1}{p}}\Biggl(
	\Ex\abs{\iint \abs{\Delta_h f(\ue)}\chi(x)\,dx\,dt}^2
	+\Ex\norm{\eta(\ue)}_{L^\infty(0,T;L^p(\chi dx))}^2 
	\Biggr).
\end{align*}
Given \eqref{eq:stoch-ue-est1} 
and \eqref{eq:eta-pol}, it follows that 
$\abs{J_1}\lesssim \abs{h}^{1-\frac{1}{p}}$, 
$p\in [2,\infty)$.

Similarly, 
$\abs{J_3}\lesssim \abs{h}^{1-\frac{1}{p}}$, 
$p\in [2,\infty)$.

\smallskip

\noindent\underline{Estimate of $J_2, J_4$.} 
Integration by parts gives
\begin{align*}
	J_2 & =\Ex\iint \eps \px \chi \Delta_h \px\ue 
	\left(\int_x^\infty \chi 
	\Delta_h \eta(\ue)\,dy\right)\,dx\,dt
	\\ & \qquad 
	-\Ex\iint \eps \Delta_h \px\ue 
	\Delta_h \eta(\ue) \chi^2\,dx\,dt
	=: J_{2,1}+J_{2,2}.
\end{align*}
By \eqref{eq:anti-der-chi-eta}, 
$\abs{\px \chi(x)}\lesssim \chi(x)$, 
and the Cauchy--Schwarz/Young product inequalities, 
\begin{align*}
	\abs{J_{2,1}} & \lesssim_\chi 
	\Ex\Biggl[ \left(\iint \eps \abs{\px \ue}^2\chi(x)\,dx\,dt
	\right)^{1/2}\left(\abs{h}^{1-\frac{1}{p}}
	\norm{\eta(\ue)}_{L^\infty(0,T;L^p(\chi dx))}
	\right)\Biggr]
	\\ & \lesssim 
	\abs{h}^{1-\frac{1}{p}} \Biggl(
	\Ex\iint \eps \abs{\px \ue}^2 \chi(x)\,dx\,dt
	+\Ex\norm{\eta(\ue)}_{L^\infty(0,T;L^p(\chi dx))}^2 
	\Biggr).
\end{align*}
Making use of \eqref{eq:stoch-ue-est1} 
and \eqref{eq:eta-pol}, 
$\abs{J_{2,1}}\lesssim \abs{h}^{1-\frac{1}{p}}$, 
$p\in [2,\infty)$.

Regarding $J_{2,2}$, let us first estimate 
$\chi\Delta_h \eta(\ue)$: 
\begin{align*}
	& \abs{\chi(x)\Delta_h \eta(\ue(t,x))}
	=\abs{\int_0^{\abs{h}}\partial_{x'} 
	\eta(\ue(t,x+x')) \chi(x)\dx'}
	\\ & \qquad \overset{\eqref{eq:weight-property}}{\lesssim}
	\int_0^{\abs{h}}\abs{\eta'(\ue(t,x+x'))}
	\abs{\partial_{x'}\ue(t,x+x'))}\chi(x+x')\dx'
	\\ & \qquad \leq 
	\left(\int_0^{\abs{h}}
	\abs{\eta'(\ue(t,x+x'))}^2\chi(x+x')\dx'\right)^{1/2}
	\\ & \qquad \qquad\qquad\times 
	\left(\int_0^{\abs{h}} \abs{\partial_{x'} 
	\ue(t,x+x')}^2 \chi(x+x')\dx'\right)^{1/2},
\end{align*}
where, setting $A:= \left(\int_0^{\abs{h}}
	\abs{\eta'(\ue(t,x+x'))}^2\chi(x+x')\dx'\right)^{1/2}$,
\begin{align*}
	& A \lesssim_\chi \abs{h}^{\frac{p-1}{2p}}
	\left(\int_0^{\abs{h}}
	\abs{\eta'(\ue(t,x+x'))}^{2p}\chi(x+x')\dx'
	\right)^{1/(2p)}
	\\ & \quad 
	\overset{\eqref{eq:eta-pol}}{\lesssim}
	\abs{h}^{\frac{p-1}{2p}}\left(1+
	\norm{\ue}_{L^\infty(0,T;L^{q_x}(\chi dx))}^{q_t}\right),
\end{align*}
for some finite $q_t,q_x$, depending on $2p$ 
and $p_0-1$, cf.~\eqref{eq:eta-pol}. Therefore,
\begin{equation}\label{eq:J22-est}
	\begin{split}
		&\abs{J_{2,2}} \lesssim \abs{h}^{\frac12-\frac{1}{2p}}
		\Ex \Biggl[\left(1
		+\norm{\ue}_{L^\infty(0,T;L^{q_x}(\chi dx))}^{q_t}\right)
		\\ & \qquad \times
		\underbrace{\iint \eps \Delta_h\px\ue
		\left(\int_0^{\abs{h}} \abs{\partial_{x'}
		\ue(t,x+x')}^2\chi(x+x')\dx'\right)^{1/2} 
		\chi\dx\dt}_{=:B}\Biggr],
	\end{split}	
\end{equation}
where, by the Cauchy-Schwarz inequality,
\begin{align*}
	& B \lesssim 
	\norm{\sqrt{\eps}\Delta_h\px\ue}_{L^2(0,T;L^2(\chi dx))}
	\left(\int_0^{\abs{h}}\!\!\!\iint\eps
	\abs{\partial_{x'}\ue(t,x+x')}^2\chi(x)
	\dx\dt\dx'\right)^{1/2}
	\\ & \quad 
	\leq \sqrt{\abs{h}}\norm{\sqrt{\eps}
	\Delta_h\px\ue}_{L^2(0,T;L^2(\chi dx))}
	\norm{\sqrt{\eps}\px\ue}_{L^2(0,T;L^2(\chi dx))},
\end{align*}
using again \eqref{eq:weight-property} to replace $\chi(x)$ by 
a constant times $\chi(x+x')$. 
Hence, for any finite $p_1,p_2,p_3$ such that 
$\frac{1}{p_1}+\frac{1}{p_2}+\frac{1}{p_3}=1$, the 
generalised H\"older inequality yields
\begin{equation}\label{eq:gen-Holder}
	\begin{split}		
		\abs{J_{2,2}} 
		\lesssim \abs{h}^{1-\frac{1}{2p}}
		\left(\Ex \abs{J_{2,1}^{(1)}}^{p_1}
		\right)^{\frac{1}{p_1}} 
		\left(\Ex\abs{J_{2,2}^{(1)}}^{p_2}
		\right)^{\frac{1}{p_2}}
		\left(\Ex\abs{J_{2,3}^{(1)}}^{p_3}
		\right)^{\frac{1}{p_3}},
	\end{split}
\end{equation}
where $J_{2,2}^{(1)}=\left(1+
\norm{\ue}_{L^\infty(0,T;L^{q_x}(\chi dx))}^{q_t}\right)$, 
$J_{2,2}^{(2)}=\norm{\sqrt{\eps}
\Delta_h\px\ue}_{L^2(0,T;L^2(\chi dx))}$, 
and $J_{2,2}^{(3)}=\norm{\sqrt{\eps} \px\ue}_{L^2(0,T;L^2(\chi dx))}$.
By \eqref{eq:stoch-ue-est1}, we conclude that
\begin{equation}\label{eq:J22-est2}
	\abs{J_{2,2}}\lesssim \abs{h}^{1-\frac{1}{2p}}.
\end{equation}
Summarising,
$\abs{J_2}\lesssim \abs{h}^{1-\frac{1}{p}}+
\abs{h}^{1-\frac{1}{2p}} 
\lesssim \abs{h}^{1-\frac{1}{p}}$, $p\in [2,\infty)$.

Similarly, integration by parts gives
\begin{align*}
	J_4 & =\Ex\iint \bigl(\eps \px \chi 
	\Delta_h \px\eta(\ue)\bigr)(t,y) 
	\left(\int_{-\infty}^y \chi 
	\Delta_h \ue\,dx\right)\,dy\,dt
	\\ & \qquad 
	-\Ex\iint \eps \Delta_h \px\eta(\ue) 
	\Delta_h \ue \chi^2\,dy\,dt
	=: J_{4,1}+J_{4,2}.
\end{align*}
By slightly modifying the calculation leading up 
to \eqref{eq:anti-der-chi-eta},
\begin{equation}\label{eq:anti-der-chi-u}
	\abs{\int_{-\infty}^y \chi(x)\Delta_h \ue(t,x)\dx}
	\lesssim \abs{h}^{1-\frac{1}{p}}
	\norm{\ue}_{L^\infty(0,T;L^p(\chi dx))},
\end{equation}
for any finite $p\ge 2$, and so, by the
Cauchy--Schwarz inequality,
\begin{align*}
	&\abs{J_{4,1}}
	\lesssim \sqrt{\eps}
	\Ex\Biggl[
	\iint \abs{\eta'(\ue)}
	\sqrt{\eps} \abs{\px\ue}\chi \dy\dt 
	\left( \abs{h}^{1-\frac{1}{p}}
	\norm{\ue}_{L^\infty(0,T;L^p(\chi dx))}
	\right)\Biggr]
	\\ & \lesssim 
	\abs{h}^{1-\frac{1}{p}}
	\Ex\Biggl[\norm{\eta'(\ue)}_{L^2(0,T;L^2(\chi dx))}
	\!\norm{\sqrt{\eps} \px\ue}_{L^2(0,T;L^2(\chi dx))}
	\! \norm{\ue}_{L^\infty(0,T;L^p(\chi dx))}
	\Biggr].
\end{align*}
Arguing as in \eqref{eq:gen-Holder} (via the 
generalised H\"older inequality) and 
using \eqref{eq:stoch-ue-est1} and 
\eqref{eq:eta-pol}, we arrive at
$\abs{J_{4,1}}\lesssim \abs{h}^{1-\frac{1}{p}}$. 

Regarding $J_{4,2}$, let us first 
do one more integration by parts:
\begin{align*}
	J_{4,2} & =\underbrace{\Ex\iint \eps \Delta_h \eta(\ue) 
	\Delta_h \px \ue \chi^2\dy\dt}_{=:J_{4,2}^{(1)}}
	+\underbrace{\Ex\iint \eps \Delta_h \eta(\ue) 
	\Delta_h \ue 2\chi \px \chi\dy\dt}_{=:J_{4,2}^{(2)}}.
\end{align*}
Note that $J_{4,2}^{(1)}$ is of the same 
form as $J_{2,2}$, and by slightly modifying 
the calculations leading up to 
\eqref{eq:J22-est}, \eqref{eq:J22-est2}, we obtain
$\abs{J_{4,2}^{(1)}}\lesssim \abs{h}^{1-\frac{1}{p}}$. 
For $J_{4,2}^{(2)}$, noting that $\abs{\Delta_h 
\ue(t,y)}\leq \int_0^{\abs{h}}\abs{\px \ue(t,y+y')}\dy'$ 
and $\abs{\px \chi(y)}\lesssim \chi(y)$, 
we proceed as follows:
\begin{align*}
	\abs{J_{4,2}^{(2)}}\,\,
	& \overset{\eqref{eq:weight-property}}{\lesssim}
	\sqrt{\eps}\Ex \int_0^{\abs{h}}\!\!\! 
	\iint \abs{\chi\Delta_h \eta(\ue)} 
	\abs{\sqrt{\eps} \px \ue(t,y+y')\chi(y+y')}
	\dy\dt\dy'
	\\ & \lesssim_{\chi}
	\sqrt{\eps}\Ex \int_0^{\abs{h}}
	\norm{\Delta_h \eta(\ue)}_{L^2(0,T;L^2(\chi dx))}
	\norm{\sqrt{\eps} \px \ue}_{L^2(0,T;L^2(\chi dx))}\dy'
	\\ & \lesssim \abs{h} \left( 
	\Ex\norm{\Delta_h \eta(\ue)}_{L^2(0,T;L^2(\chi dx))}^2 
	+\Ex\norm{\sqrt{\eps} \px \ue}_{L^2(0,T;L^2(\chi dx))}^2\right).
\end{align*}
Given \eqref{eq:stoch-ue-est1} and \eqref{eq:eta-pol}, 
we conclude that $\abs{J_{4,2}^{(2)}}\lesssim \abs{h}$
and thus $\abs{J_{4,2}}\lesssim 
\abs{h}^{1-\frac{1}{p}}$. 

Summarising,
$\abs{J_4}\lesssim \abs{h}^{1-\frac{1}{p}}$, $p\in [2,\infty)$.

\smallskip

\noindent\underline{Estimate of $J_5$.} 
We have
\begin{align*}
	\abs{J_5} & \lesssim 
	\Ex\iint \chi(y) \bigl(\abs{\mu_\eps}(t,y+h)
	+\abs{\mu_\eps}(t,y)\bigr)
	\abs{\int_{-\infty}^y \chi(x)
	\Delta_h \ue(t,x)\dx}\dy\dt.
\end{align*}
By \eqref{eq:anti-der-chi-u} 
and \eqref{eq:stoch-ue-est1}, we deduce that, 
for any $p\in [2,\infty)$,
\begin{align*}
	\abs{J_5} & \lesssim 
	\abs{h}^{1-\frac{1}{p}}
	\Ex \left[\norm{\ue}_{L^\infty(0,T;L^p(\chi dx))}
	\abs{\iint \chi(y)\,d\abs{\mu_\eps}(t,y)}\right]	\\ & \lesssim
	\abs{h}^{1-\frac{1}{p}}
	\left(\Ex \norm{\ue}_{L^\infty(0,T;L^p(\chi dx)}^2
	+\Ex\abs{\iint \chi(y)\,d\abs{\mu_\eps}(t,y)}^2\right)
	\lesssim \abs{h}^{1-\frac{1}{p}}.
\end{align*}

\smallskip

\noindent\underline{Estimate of $J_6$.} 
Using \eqref{eq:anti-der-chi-u}, 
\begin{align*}
	\abs{J_6} & = \frac12 \Ex\iint \abs{\chi \Delta_h 
	\left(\eta''(\ue)\sigma^2(y,\ue)\right)}
	\abs{\int_{-\infty}^y \chi \Delta_h \ue\dx}\dy\dt
	\\ & \lesssim 
	\abs{h}^{1-\frac{1}{p}}
	\Ex\left[\iint\abs{\Delta_h 
	\left(\eta''(\ue)\sigma^2(y,\ue)\right)}\chi(y)\dy\dt
	\norm{\ue}_{L^\infty(0,T;L^p(\chi dx))}\right]
	\\ & \lesssim 
	\abs{h}^{1-\frac{1}{p}}
	\Biggl(\Ex\abs{\iint\abs{\Delta_h 
	\left(\eta''(\ue)\sigma^2(y,\ue)\right)}\chi(y)\dy\dt}^2 
	+\Ex \norm{\ue}_{L^\infty(0,T;L^p(\chi dx))}^2\Biggr).
\end{align*}
In view of \eqref{eq:stoch-ue-est1}, 
\eqref{eq:eta-pol} and \eqref{eq:sigma-lin}, we conclude that
$\abs{J_6}\lesssim \abs{h}^{1-\frac{1}{p}}$, 
$p\in [2,\infty)$.

\smallskip

\noindent\underline{Estimate of $J_7$.}
To estimate $J_7$, let us write
\begin{align*}
	& \iint\limits_{x<y} 
	\chi(x)\Delta_h\sigma(x,\ue(t,x))) 
	\chi(y)\Delta_h \bigl(\eta'(\ue(t,y)) 
	\sigma(y,\ue(t,y))\bigr)\dx\dy
	\\ & \quad 
	= \int_{\R} \left(\int_{-\infty}^y 
	\chi(x)\Delta_h\sigma(x,\ue(t,x)))\dx \right)
	\chi(y)\Delta_h
	\bigl(\eta'(\ue(t,y))\sigma(y,\ue(t,y))\bigr)\dy.
\end{align*}
Using \eqref{eq:anti-der-chi-u} with $\ue$ 
replaced by $\sigma(x,\ue)$, we obtain
\begin{align*}
	\abs{\int_{-\infty}^y 
	\chi(x)\Delta_h\sigma(x,\ue(t,x)))\,dx}
	&\lesssim \abs{h}^{1-\frac{1}{p}}
	\norm{\sigma(\cdot,\ue)}_{L^\infty(0,T;L^p(\chi dx))}
	\\ & 
	\overset{\eqref{eq:sigma-lin}}{\lesssim} 
	\abs{h}^{1-\frac{1}{p}}
	\left(1+\norm{\ue}_{L^\infty(0,T;L^p(\chi dx))}\right),
\end{align*}
which implies that 
\begin{align*}
	\abs{J_7} & \leq \abs{h}^{1-\frac{1}{p}}
	\Ex \left[\left(1+\norm{\ue}_{L^\infty(0,T;L^p(\chi dx))}\right)
	\iint \chi\Delta_h \bigl(\eta'(\ue)
	\sigma(y,\ue)\bigr)\,dy\,dt\right]
	\\ &  
	\lesssim  \abs{h}^{1-\frac{1}{p}}
	\Biggl(\Ex\abs{1+\norm{\ue}_{L^\infty(0,T;L^p(\chi dx)}}^2
	+\Ex\abs{\iint \chi \Delta_h 
	\bigl(\eta'(\ue)\sigma(y,\ue)\bigr)\,dy\,dt}^2\Biggr).
\end{align*}
Given \eqref{eq:stoch-ue-est1}, 
\eqref{eq:eta-pol} and \eqref{eq:sigma-lin}, we obtain 
$\abs{J_7}\lesssim \abs{h}^{1-\frac{1}{p}}$, $p\in [2,\infty)$.

\smallskip

\noindent\underline{Estimate of $J_8$.} For any fixed $t$, 
we write
\begin{align*}
	&\iint\limits_{x<y} \bigl(\chi\Delta_h \ue\bigr)(t,x)
	\bigl(\chi(x)\Delta_h \eta(\ue)\bigr)(t,y)\dx\dy
	\\ & \quad 
	=\int_{\R} \left(\int_{-\infty}^y
	\chi(x)\Delta_h \ue(t,x)\dx\right)
	\chi(y)\Delta_h \eta(\ue(t,y)) \dy.
\end{align*}
By combining this with \eqref{eq:anti-der-chi-u}, 
Young's product inequality, \eqref{eq:eta-pol}, 
and \eqref{eq:stoch-ue-est1}, we 
obtain the following bound for $J_8$, 
for any finite $p\ge 2$:
\begin{align*}
	\abs{J_8}\lesssim 
	\abs{h}^{1-\frac{1}{p}}
	\Ex\left[\norm{\ue}_{L^\infty(0,T;L^p(\chi dx))}
	\norm{\eta(\ue)}_{L^\infty(0,T;L^1(\chi dx))}\right]
	\lesssim \abs{h}^{1-\frac{1}{p}}.
\end{align*}

\noindent\underline{Conclusion.} Summarising, 
\eqref{eq:stoch-Delta-x} is bounded by 
$C(\chi,p) \abs{h}^{1-\frac{1}{p}}$, $p\in [2,\infty)$, 
which, upon replacing $h$ by $z$, $\eps$ by $\eps_n$, 
and $\ue$ by $u_n$, becomes \eqref{eq:stoch-CC-est}. 



\end{document}